\documentclass{amsart}
\usepackage{amssymb}
\usepackage{amsfonts}
\usepackage{amssymb}
\usepackage{amsmath}
\usepackage{amsthm}
\usepackage{enumerate}
\usepackage{tabularx}
\usepackage{centernot}
\usepackage[official]{eurosym}
\usepackage{amsthm,amssymb}

\usepackage{tikz}
\usetikzlibrary{decorations.pathmorphing,shapes}
\usetikzlibrary{trees}
\usepackage{bussproofs}

\def\prepQuaternary{%
    \ifnum\theLevel<4
        \errmessage{Hypotheses missing!}
    \fi%
    \edef\rrcurBox{\thecur{myBox}}
    \edef\rrcurScoreStart{\thecur{myScoreStart}}%
    \edef\rrcurCenter{\thecur{myCenter}}%
    \edef\rrcurScoreEnd{\thecur{myScoreEnd}}%
    \advance\theLevel by-1%
    \edef\rcurBox{\thecur{myBox}}
    \edef\rcurScoreStart{\thecur{myScoreStart}}%
    \edef\rcurCenter{\thecur{myCenter}}%
    \edef\rcurScoreEnd{\thecur{myScoreEnd}}%
    \advance\theLevel by-1%
    \edef\ccurBox{\thecur{myBox}}
    \edef\ccurScoreStart{\thecur{myScoreStart}}%
    \edef\ccurCenter{\thecur{myCenter}}%
    \edef\ccurScoreEnd{\thecur{myScoreEnd}}%
    \advance\theLevel by-1%
    \edef\lcurBox{\thecur{myBox}}
    \edef\lcurScoreStart{\thecur{myScoreStart}}%
    \edef\lcurCenter{\thecur{myCenter}}%
    \edef\lcurScoreEnd{\thecur{myScoreEnd}}%
}

\def\QuaternaryInf$#1\fCenter#2${%
    \prepQuaternary%
    \buildConclusion{#1}{#2}%
    \joinQuaternary%
    \resetInferenceDefaults%
    \ignorespaces%
}

\def\QuaternaryInfC#1{%
    \prepQuaternary%
    \buildConclusionC{#1}%
    \joinQuaternary%
    \resetInferenceDefaults%
    \ignorespaces%
}

\def\joinQuaternary{
    \setbox\myBoxA=\hbox{\theHypSeparation}%
    \lcurScoreEnd=\rrcurScoreEnd%
    \advance\lcurScoreEnd by\wd\rcurBox%
    \advance\lcurScoreEnd by\wd\lcurBox%
    \advance\lcurScoreEnd by\wd\ccurBox%
    \advance\lcurScoreEnd by3\wd\myBoxA%
    \displace=\lcurScoreEnd%
    \advance\displace by -\lcurScoreStart%
    \lcurCenter=.5\displace%
    \advance\lcurCenter by\lcurScoreStart%
    \ifx\rootAtBottomFlag\myTrue%
        \setbox\lcurBox=%
            \hbox{\box\lcurBox\unhcopy\myBoxA\box\ccurBox%
                      \unhcopy\myBoxA\box\rcurBox
                      \unhcopy\myBoxA\box\rrcurBox}%
    \else%
        \htLbox = \ht\lcurBox%
        \htAbox = \ht\myBoxA%
        \htCbox = \ht\ccurBox%
        \htRbox = \ht\rcurBox%
        \htRRbox = \ht\rrcurBox%
        \setbox\lcurBox=%
            \hbox{\lower\htLbox\box\lcurBox%
                  \lower\htAbox\copy\myBoxA\lower\htCbox\box\ccurBox%
                  \lower\htAbox\copy\myBoxA\lower\htRbox\box\rcurBox%
                  \lower\htAbox\copy\myBoxA\lower\htRRbox\box\rrcurBox}%
    \fi%
    \displace=\newCenter%
    \advance\displace by -.5\newScoreStart%
    \advance\displace by -.5\newScoreEnd%
    \advance\lcurCenter by \displace%
    \edef\curBox{\lcurBox}%
    \edef\curScoreStart{\lcurScoreStart}%
    \edef\curScoreEnd{\lcurScoreEnd}%
    \edef\curCenter{\lcurCenter}%
    \joinUnary%
}

\usepackage{tikz}
\usetikzlibrary{decorations.pathmorphing,shapes}

\renewcommand{\leq}{\leqslant}
\renewcommand{\geq}{\geqslant}

\makeatletter
\def\subsection{\@startsection{subsection}{3}%
  \z@{.5\linespacing\@plus.7\linespacing}{.3\linespacing}%
  {\bfseries\centering}}
\makeatother

\makeatletter
\def\subsubsection{\@startsection{subsubsection}{3}%
  \z@{.5\linespacing\@plus.7\linespacing}{.3\linespacing}%
  {\centering}}
\makeatother

\makeatletter
\def\myfnt{\ifx\protect\@typeset@protect\expandafter\footnote\else\expandafter\@gobble\fi}
\makeatother

\theoremstyle{definition}

\newtheorem{theorem}{Theorem}[section]
\newtheorem{definition}[theorem]{Definition}

\newtheorem{proposition}[theorem]{Proposition}
\newtheorem{example}[theorem]{Example}
\newtheorem{corollary}[theorem]{Corollary}

\newcounter{claimcounter}
\numberwithin{claimcounter}{theorem}

\newcommand*\dep{{=\mkern-1.2mu}}

\def\presuper#1#2%
  {\mathop{}%
   \mathopen{\vphantom{#2}}^{#1}%
   \kern-\scriptspace%
   #2}

\begin{document}

\title{A Finite Axiomatization of G-Dependence}
\thanks{The research of the author was supported by the Finnish Academy of Science and Letters (Vilho, Yrj\"o and Kalle V\"ais\"al\"a foundation). The author would like to thank Jouko V\"a\"an\"anen for introducing him to the topic of G-dependence and Tapani Hyttinen for an illuminating remark.}

\author{Gianluca Paolini}
\address{Department of Mathematics and Statistics,  University of Helsinki, Finland}  

\begin{abstract} We show that a form of dependence known as G-dependence (originally introduced by Grelling) admits a very natural finite axiomatization, as well as Armstrong relations. We also give an explicit translation between functional dependence and G-dependence.
\end{abstract}

\maketitle

\section{Introduction}


	In this paper we study a form of dependence originally introduced by Grelling \cite{grelling}, and recently analyzed by V\"a\"an\"anen \cite{vaananen_G_dep} in the context of database dependency theory and team semantics. The main objective of the paper is to provide a finite axiomatization of G-dependence, and to give an explicit translation between functional dependence and G-dependence. 
	
	First of all, let us explain this notion of dependence, and where it comes from. In \cite{grelling} Grelling lays the foundations of a mathematical theory of dependence, isolating several notions of dependence and noticing various properties that they satisfy. Among them, there is the notion of functional dependence, later axiomatized by Armstrong in the context of database theory (see \cite{armstrong}), and other notions of dependence familiar in combinatorics, specifically in the study of closure operators (see e.g. \cite{pregeo}). 
	Grelling also introduces a notion of dependence that fails to fit with any of the currently known forms of dependence studied in database dependency theory and combinatorics. He calls it {\em varequidependence}, but we will refer to it simply as {\em G-dependence} (G for Grelling). 
	From the perspective of database theory, G-dependence can be described as follows: a set of attributes $y$ G-depends on a set of attributes $x$ in the database $X$ if for any two rows $s, s'$ of $X$ for which exactly one $x_i \in x$ has different values, there exists at least one $y_j \in y$ such that $y_j$ has different values in $s$ and $s'$.
	
		To give an example of G-dependence consider Table \ref{salary}, where we see an example of a database storing informations about salaries of the (unfortunately not so hypothetical) university $U$. 
\begin{figure}[h]
$$\begin{array}{|c|c|c|c|}
\hline
\textbf{Name} & \textbf{Title} & \textbf{Years of Experience} & \textbf{Salary} \\
\hline
	 \text{John}  & \text{PhD} & \text{1} & \text{\euro{2200}} \\
	 \text{Marie} & \text{PhD} & \text{10}  & \text{\euro{2200}} \\
	 \text{Paolo} & \text{Professor} & \text{5} & \text{\euro{3500}} \\
	 \text{Sara}  & \text{Professor} & \text{7} & \text{\euro{3500}} \\
\hline
\end{array}
$$\caption{Salaries of university $U$  \label{salary}}\end{figure}
	First, we recall the definition of functional dependence, to compare the two notions: we say that a set of attributes $y$ functionally depends on a set of attributes $x$ in the database $X$ if for any two rows $s, s'$ of $X$ if $s$ and $s'$ agree on $x$, then they also agree on $y$. 
Now, according to this definition, the salary of an academic in the university $U$ depends on his/her academic title, and also on the combination of his/her academic title and his/her years of experience. This sounds very reasonable, what a good university university $U$ is, one would think. But if we look at the table we actually notice that both Dr. Marie and Dr. John have the same salary despite John has only 1 year of experience and Marie has 10! Thus, the claimed dependence of salary from academic title and years of experience seems very counterintuitive. On the other hand, if we appeal to G-dependence we see that this is not the case, since, letting the academic title fixed, more years of experience do {\em not} imply an increase in salary in the salary system of university $U$. One of the virtues of G-dependence is the solution of this pathology of functional dependence. G-dependence is in fact a form of dependence that respects very closely our intuitions about dependence. Given its intuitive appeal we want to analyze G-dependence in a mathematical fashion, using the tools from database dependency theory and team semantics. 
	
	We first give some background. Team semantics is a new logical framework, originally introduced by Hodges \cite{hodges} and later developed by V\"a\"an\"anen \cite{vaananen}, which aims at the construction of a logical theory of various phenomena such as (in)dependence, uncertainty and information. The common denominator of this line of research is the evaluation of logical formulas in function of so-called {\em teams}, i.e. {\em sets (or multisets)} of assignments to a given set of variables. This widens the scope of traditional logical systems, allowing for the modelling of new non-classical phenomena. For example, the use of several assignments allows to give meaning to logical atoms expressing various dependency relations among variables, and the consequent construction of extensions of first-order and propositional logic. This particular line of research is known as {\em dependence logic} \cite{vaananen}. Another scope of application of team semantics is in the realm of (quantum) information theory \cite{QTL} and probability logic \cite{measure_teams}. In fact, the multiplicity of assignments under which we evaluate formulas allows for the introduction of various probabilistic notions and the consequent construction of interesting logics. 
	
	Soon after its birth, dependence logic was brought to the rediscovery of a rich field of research active since the '70s, known as {\em database dependency theory} \cite{dat_book}. This part of database theory aims primarily at the analysis and axiomatization of the various forms of constraints that a database can be subject to. The fundamental connection between database theory and teams semantics is that any database can be immediately considered as a set of assignments (i.e. a team) and the other way around. Consequently, database dependency theory and dependence logic overlap significantly in many respects, and can be considered as a general framework for the study of dependence.

	We will then use this established technology to study our object of interest, i.e. G-dependence. In this paper will give a complete axiomatization of G-dependence, in the style of Armstrong's axiomatization of functional dependence \cite{armstrong}, a problem left open in V\"a\"an\"anen's treatment of G-dependence \cite{vaananen_G_dep}. We will also prove that G-dependence admits Armstrong relations, in the sense of \cite{armstrong}, \cite{beeri} and \cite{fagin}, and provide an explicit translation between G-dependence and functional dependence. Finally, in the style of the dependence logic of \cite{vaananen}, we will build an extension of first-order logic called G-dependence logic and use the above mentioned translation between G-dependence and functional dependence to conclude that G-dependence logic has the same expressive power of existential second-order logic.

\section{G-Dependence}\label{Sec_2}

	In this section we define what a team is, and see how to define G-dependence. We then define a logical language made of atomic expressions called {\em G-dependence atoms}, and provide a set of rules and axioms for this language under the intended semantics of G-dependence.

\smallskip
	
	Let $\mathrm{Var}$ be a countable set of symbols, called {\em attributes} or {\em individual variables}.
	
	\begin{definition}[Team]\label{team} A {\em team} $X$ with values in a set $A$ and domain $\mathrm{dom}(X) \subseteq \mathrm{Var}$ is a pair $(I, \tau)$ such that $I$ is a set and $\tau: I \rightarrow A^{\mathrm{dom}(X)}$ is a function.
\end{definition}

	Notice that a team is nothing but a database with multiple rows, i.e. the same tuple is allowed to occur several times. Elsewhere in the literature (e.g. \cite{measure_teams}) what we call a team is referred to as a multiteam. We only talk of teams here for simplicity. Given a team $X = (I, \tau)$ and $s \in I$ we denote the value $\tau(s)(x)$ simply as $x(s)$.
	
	\begin{example} In Figure \ref{dis} we see an example of a team $X = (\left\{ 0, 1, 2, 3 \right\}, \tau)$ with values in $2 = \left\{ 0, 1 \right\}$ and domain $\left\{ x_0, x_1, y_0, y_1 \right\}$. The values $x_i(s)$ and $y_i(s)$ are evident from the table, e.g. $x_0(0) = 0$, $x_1(3) = 0$ and $y_0(1) = 1$.
\begin{figure}[h]
$$\begin{array}{|c|c|c|c|c|}
\hline
\phantom{a} & x_0 & x_1 & y_0 & y_1 \\
\hline
	0 & 0 & 0 & 0 & 0 \\
	1 & 1 & 0 & 1 & 0 \\
	2 & 0 & 0 & 0 & 0 \\
	3 & 0 & 0 & 0 & 0 \\
\hline
\end{array}
$$\caption{Example of a team  \label{dis}}\end{figure}
\end{example}	

	\begin{definition}\label{G_dep} Let $X = (I, \tau)$ be a team and $x$ and $y$ finite sets of variables from $\mathrm{dom}(X)$. We say that $y$ G-depends on $x$, in symbols $X \models  {{\Bumpeq}\mkern-2mu}(x, y)$, if for every $s, s' \in I$, if there exists exactly one $x_i \in x$ such that $x_i(s) \neq x_i(s')$, then there exists at least one $y_j \in y$ such that $y_j(s) \neq y_j(s')$.
\end{definition}

	We will refer to formulas of the form ${{\Bumpeq}\mkern-2mu}(x, y)$ as in Definition \ref{G_dep}, as G-dependence atoms. We stress here that in expressions of the kind ${{\Bumpeq}\mkern-2mu}(x, y)$ we consider $x$ and $y$ as {\em sets} of variables, not sequences. Following this convention, we use expressions like ${{\Bumpeq}\mkern-2mu}(x, yz)$ to mean ${{\Bumpeq}\mkern-2mu}(x, y \cup z)$. Furthermore, if in the formula ${{\Bumpeq}\mkern-2mu}(x, y)$ the set $x$ is a singleton $\{ x_0 \}$ we simply write ${{\Bumpeq}\mkern-2mu}(x_0, y)$.

	\begin{example} As an example, for the team $X$ represented in Figure \ref{dis} we have:
	$$ X \models  {{\Bumpeq}\mkern-2mu}(x_0, y_0), \;\; X \models  {{\Bumpeq}\mkern-2mu}(x_0x_1, y_0), \;\; X \models {{\Bumpeq}\mkern-2mu}(x_0y_0, y_1), \;\; X \not\models {{\Bumpeq}\mkern-2mu}(x_0x_1, y_1).$$
\end{example}

	Notice that in analogy with functional dependence and similar forms of dependence, the G-dependence atom is {\em local}, i.e. the truth of $X \models  {{\Bumpeq}\mkern-2mu}(x, y)$ depends only on the values of the variables occurring in $x$ or $y$.

	As usual, we let:
	\begin{enumerate}[i)]
	\item $X \models \Sigma$ if $X \models \sigma$ for every $\sigma \in \Sigma$;
	\item $\Sigma \models \sigma$ if, for every team $X$, $X \models \Sigma$ implies $X \models \sigma$.
\end{enumerate}

\medskip	

	The goal of the present paper is to find a complete axiomatization of the relation $\Sigma \models \sigma$, where $\Sigma$ contains $G$-dependence atoms.  I.e. we want to find a fundamental set of principles (rules) that govern the properties of the G-dependence atoms, and we want to find them so that they are complete. That is, all the G-dependence atoms that semantically follow from the given ones, also deductively follow from them, i.e. follow from them according to the rules. 
	
	The deductive system for G-dependence atoms consists of the following six rules\footnote{The isolation of rules ($R_0$) and ($R'_0$) is due to Tapani Hyttinen.}, where for an atom ${{\Bumpeq}\mkern-2mu}(x, y)$ we let $x = \left\{ x_0, ..., x_{n-1} \right\}$, and, for $i < n$, we let $(x - x_i) y = \left\{ x_j \in x : j \neq i \right\} \cup y$.

\bigskip

\begin{prooftree}
\AxiomC{}
\LeftLabel{($A_0$)}
\UnaryInfC{${{\Bumpeq}\mkern-2mu}(\emptyset, y)$}
\end{prooftree}

\medskip

\begin{prooftree}
\AxiomC{}
\LeftLabel{($A_1$)}
\UnaryInfC{${{\Bumpeq}\mkern-2mu}(x, x)$}
\end{prooftree}

\medskip
	
\begin{prooftree}
\AxiomC{${{\Bumpeq}\mkern-2mu}(x, y)$}
\LeftLabel{($R_0$)}
\UnaryInfC{${{\Bumpeq}\mkern-2mu}(x_i, (x - x_i) y)$}
\end{prooftree}

\medskip

\begin{prooftree}
\AxiomC{${{\Bumpeq}\mkern-2mu}(x_0, (x-x_0) y)$}
\AxiomC{$\cdots$}
\AxiomC{${{\Bumpeq}\mkern-2mu}(x_{n-1}, (x-x_{n-1}) y)$}
\LeftLabel{($R'_0$)}
\TrinaryInfC{${{\Bumpeq}\mkern-2mu}(x_0 \cdots x_{n-1}, y)$}
\end{prooftree}

\medskip

\begin{prooftree}
\AxiomC{${{\Bumpeq}\mkern-2mu}(x, y)$}
\LeftLabel{($R_1$)}
\UnaryInfC{${{\Bumpeq}\mkern-2mu}(x, yz)$}
\end{prooftree}

\medskip

\begin{prooftree}
\AxiomC{${{\Bumpeq}\mkern-2mu}(z_0, y)$}
\AxiomC{$\cdots$}
\AxiomC{${{\Bumpeq}\mkern-2mu}(z_{n-1}, y)$}
\AxiomC{${{\Bumpeq}\mkern-2mu}(x_0, z_0 \cdots z_{n-1})$}
\LeftLabel{($R_2$)}
\QuaternaryInfC{${{\Bumpeq}\mkern-2mu}(x_0, y)$}
\end{prooftree}

\bigskip	

	We refer to rules with an empty set of premises as axioms, and consequently reserve the term rule to rules with at least one premise. A deduction with premises on a set of atoms $\Sigma$ is a sequence of atoms $(\sigma_0 , ... , \sigma_{n-1})$ such that each $\sigma_i$ is either an instance of the axioms of our deductive system, or an atom in $\Sigma$, or follows from one or more atoms of $\left\{ \sigma_0, ... , \sigma_{i-1} \right\}$ by one of its rules. We say that $\sigma$ is provable from $\Sigma$, in symbols $\Sigma \vdash \sigma$, if there is a deduction  $(\sigma_0 , ... , \sigma_{n-1})$ with premises on $\Sigma$ and $\sigma = \sigma_{n-1}$.

\section{Completeness}

	In this section we prove the main results of the paper, i.e. the completeness of the deductive system introduced in Section \ref{Sec_2} and the admissibility of the so-called Armstrong relations.
	
\smallskip

	We denote the set of natural numbers by $\omega$.
	
	\begin{theorem}\label{compl} Let $\Sigma \cup \left\{ \sigma \right\}$ be a set of G-dependence atoms. Then 
	$$\Sigma \vdash \sigma \;\; \text{ if and only if } \;\; \Sigma \models \sigma$$ 
I.e. the G-dependence deductive system is sound and complete.
\end{theorem}

	\begin{proof} Regarding soundness, the soundness of $(A_0)$, $(A_1)$ and $(R_1)$ is clear. We prove the soundness of $(R_0)$. Let $X = (I, \tau) \not\models {{\Bumpeq}\mkern-2mu}(x_i, (x - x_i) y)$, then there exist $s, s' \in I$ such that $x_i(s) \neq x_i(s')$ and $z_j(s) = z_j(s')$ for every $z_j \in (x - x_i) y$. Thus, $X \not\models {{\Bumpeq}\mkern-2mu}(x, y)$, since $x_i(s) \neq x_i(s')$, all the $x_j$ different than $x_i$ are so that $x_j(s) = x_j(s')$ and furthermore for every $y_j \in y$ we have $y_j(s) = y_j(s')$.

We prove the soundness of $(R'_0)$. Let $X = (I, \tau) \not\models {{\Bumpeq}\mkern-2mu}(x, y)$, then there exist $s, s' \in I$ and $x_i \in x$ such that $x_i(s) \neq x_i(s')$, all the $x_j$ different than $x_i$ are so that $x_j(s) = x_j(s')$ and furthermore for every $y_j \in y$ we have $y_j(s) = y_j(s')$. Thus, $X \not\models {{\Bumpeq}\mkern-2mu}(x_i, (x - x_i) y)$, since $x_i(s) \neq x_i(s')$ but $z_j(s) = z_j(s')$ for every $z_j \in (x - x_i) y$.

Finally, we prove the soundness of ($R_2$). Let $X = (I, \tau) \not\models {{\Bumpeq}\mkern-2mu}(x_0, y)$, then there exist $s, s' \in I$ such that $x_0(s) \neq x_0(s')$ and $y_j(s) = y_j(s')$ for every $y_j \in y$. If there exists $i < n$ such that $z_i(s) \neq z_i(s')$, then $X \not\models {{\Bumpeq}\mkern-2mu}(z_i, y)$. If, on the other hand, $z_i(s) = z_i(s')$ for every $i < n$, then $X \not\models {{\Bumpeq}\mkern-2mu}(x_0, z_0 \cdots z_{n-1})$.
	
Regarding completeness, suppose that $\Sigma \nvdash \sigma$, and let 
$$\sigma = {{\Bumpeq}\mkern-2mu}(x_0 \cdots x_{n-1}, y_0 \cdots y_{m-1}) = {{\Bumpeq}\mkern-2mu}(x, y).$$
	We want to build a team $X$ with domain $\mathrm{Var}$ and values in $2 = \left\{ 0, 1 \right\}$ such that $X \models \Sigma$ and $X \not\models \sigma$. W.l.o.g. we can assume that $\Sigma$ is deductively closed, i.e. if $\Sigma \vdash \pi$, then $\pi \in \Sigma$. Because of axiom ($A_0$) we must have that $n \geq 1$. Furthermore, 
there has to be $i < n$ such that ${{\Bumpeq}\mkern-2mu}(x_i, (x-x_i) y) \not\in \Sigma$ and $x_i \not\in y$. First of all, suppose that ${{\Bumpeq}\mkern-2mu}(x_i, (x-x_i) y) \in \Sigma$, for every $i < n$. Then, by ($R'_0$), we have ${{\Bumpeq}\mkern-2mu}(x, y) \in \Sigma$, contradicting the assumption. Let then $i < n$ such that ${{\Bumpeq}\mkern-2mu}(x_i, (x-x_i) y) \not\in \Sigma$, and suppose that $x_i \in y$, then by consecutively applying ($A_1$) and ($R_1$), we obtain ${{\Bumpeq}\mkern-2mu}(x_i, (x-x_i) y) \in \Sigma$, which is contradictory.
	Thus, there exists $i < n$ such that ${{\Bumpeq}\mkern-2mu}(x_i, (x-x_i) y) \not\in \Sigma$ and $x_i \not\in y$. W.l.o.g. ${{\Bumpeq}\mkern-2mu}(x_0, (x-x_0) y) \not\in \Sigma$ and $x_0 \not\in y$. We want to define $V_0, V_1 \subseteq \mathrm{Var}$ such that $V_0 \cap V_1 = \emptyset$ and $V_0 \cup V_1 = \mathrm{Var}$, to be used in the construction of the team $X$. Let then $\left\{ v_i : i < \omega \right\} = \mathrm{Var}$, and for $i < \omega$ let:
		\[ \begin{cases} v_i \in V_0 \;\;\;\;\; \text{ if } {{\Bumpeq}\mkern-2mu}(v_i, x_1 \cdots x_{n-1}y) \not\in \Sigma \\
						  v_i \in V_1 \;\;\;\;\; \text{ otherwise. } \end{cases} \]
Notice that $x_0 \in V_0$, since we are assuming that ${{\Bumpeq}\mkern-2mu}(x_0, (x-x_0) y) \not\in \Sigma$. Furthermore, by ($A_1$) and ($R_1$), $\left\{ x_1, ..., x_{n-1}, y_0, ..., y_{m-1} \right\} \subseteq V_1$.
Let now $X = (2, \tau)$ be the team such that for every $v \in \mathrm{Var}$
\[ \begin{cases} v(0) = 0 \text{ and } v(1) = 1 \;\;\;\;\; \text{ if } v \in V_0 \\
						  v(0) = v(1) = 0 \,\;\;\;\;\;\;\;\;\;\;\;\;\;\;\, \text{ if } v \in V_1. \end{cases} \]
That is, $X$ looks as in Figure \ref{flow}.
	\begin{figure}[h]
$$\begin{array}{|c|c|c|c|c|c|c|c|c|c|}
\hline
& x_0 & x_1 & \cdots & x_{n-1} & y_0 & \cdots & y_{m-1} & V_0 & V_1 \\
\hline
0 & 0   & 0 & \cdots   & 0       & 0 & \cdots   & 0       & 0    & 0    \\
\hline
1 & 1	& 0 & \cdots   & 0       & 0 & \cdots   & 0       & 1    & 0    \\
\hline
\end{array}
$$\caption{The team $X$  \label{flow}}\end{figure}	
Evidently, $X \not\models \sigma$ (recall that $x_0 \not\in y$). We want to show that $X \models \Sigma$. Because of the rules ($R_0$) and ($R'_0$), it suffices to show that for every $\sigma' = {{\Bumpeq}\mkern-2mu}(z_0, w_0 \cdots w_{k-1}) \in \Sigma$ we have $X \models \sigma'$. In fact, suppose this is the case and let ${{\Bumpeq}\mkern-2mu}(x', y') = {{\Bumpeq}\mkern-2mu}(x'_0 \cdots x'_{k-1}, y') \in \Sigma$, we want to show that $X \models  {{\Bumpeq}\mkern-2mu}(x', y')$. By rule ($R_0$), for every $i < k$ we have ${{\Bumpeq}\mkern-2mu}(x'_i, (x'-x'_i)y') \in \Sigma$, since $\Sigma$ is deductively closed. Thus, for every $i < k$, we have $X \models {{\Bumpeq}\mkern-2mu}(x'_i, (x'-x'_i)y')$, since we are assuming that for all $\sigma' = {{\Bumpeq}\mkern-2mu}(z_0, w_0 \cdots w_{k-1}) \in \Sigma$ we have $X \models \sigma'$. Hence, by the soundness of rule ($R'_0$), we can conclude that $X \models {{\Bumpeq}\mkern-2mu}(x', y')$, as wanted.  Let then $\sigma' = {{\Bumpeq}\mkern-2mu}(z_0, w_0 \cdots w_{k-1}) \in \Sigma$. If $z_0 \in V_1$ or $\left\{ w_0, ..., w_{k-1} \right\} \cap V_0 \neq \emptyset$, then of course $X \models \sigma'$. Suppose then that $z_0 \in V_0$ and  $\left\{ w_0, ..., w_{k-1} \right\} \subseteq V_1$. Then for every $i < k$ we have ${{\Bumpeq}\mkern-2mu}(w_i, x^*y) \in \Sigma$, where $x^* = \left\{ x_1, ..., x_{n-1} \right\}$. Thus, because of ($R_2$), we have:
\begin{prooftree}
\AxiomC{$\Sigma$}
\UnaryInfC{${{\Bumpeq}\mkern-2mu}(w_0, x^*y)$}
\AxiomC{$\Sigma$}
\UnaryInfC{$\phantom{(i)} \cdots \phantom{(i)}$}
\AxiomC{$\Sigma$}
\UnaryInfC{${{\Bumpeq}\mkern-2mu}(w_{k-1}, x^*y)$}
\AxiomC{$\Sigma$}
\UnaryInfC{${{\Bumpeq}\mkern-2mu}(z_0, w_0 \cdots w_{k-1})$}
\LeftLabel{($R_2$)}
\QuaternaryInfC{${{\Bumpeq}\mkern-2mu}(z_0, x^*y)$}
\end{prooftree}

\smallskip
\noindent And so ${{\Bumpeq}\mkern-2mu}(z_0,  x^*y) \in \Sigma$, which implies that $z_0 \not\in V_0$, a contradiction.
\end{proof}

	Using the ``disjoint union" technique of Beeri et al. (cf. \cite{beeri} and \cite{fagin}) it is now immediate to conclude the existence of Armstrong relations for G-dependence.

	\begin{theorem}[Existence of Armstrong Relations] Let $\Sigma$ be a set of G-dependence atoms. Then there exists a team $X$, called an Armstrong relation, such that
	$$X \models \sigma \text{\; iff \;} \Sigma \vdash \sigma.$$
\end{theorem}

	\begin{proof} W.l.o.g. $\Sigma$ is deductively closed. Let $\Sigma^*$ be the set of G -dependence atoms not in $\Sigma$ and $(\sigma_{2i})_{i < k \leq \omega}$ an enumeration of $\Sigma^*$. For every $\sigma = \sigma_{2i} \in \Sigma^*$ we have that $\Sigma \nvdash \sigma$, and so by the proof of Theorem \ref{compl}, we can find $X_{2i} = (\left\{ 2i, 2i + 1 \right\}, \tau_{2i})$, with $\tau_{2i}: \left\{ 2i, 2i + 1 \right\} \rightarrow \left\{ 2i, 2i + 1 \right\}^{\mathrm{Var}}$, such that $X_{2i} \models \Sigma$ and $X_{2i} \not\models \sigma$. Let $n = 2(k+1)$ if $k < \omega$, and $n = \omega$ otherwise. Let $X = (\left\{ i: i < n \right\}, \tau)$ be such that, for $i < k$ and $v \in \mathrm{Var}$, we let
	\[\tau(2i)(v) = \begin{cases} \tau_{2i}(2i)(v) \;\; \text{ if } {{\Bumpeq}\mkern-2mu}(v, \emptyset) \not\in \Sigma \\
						          0 \;\;\;\;\;\;\;\;\;\;\;\;\;\;\; \text{ otherwise } \end{cases}\]
and
	\[\tau(2i+1)(v) = \begin{cases} \tau_{2i}(2i+1)(v) \;\; \text{ if } {{\Bumpeq}\mkern-2mu}(v, \emptyset) \not\in \Sigma \\
						          0 \;\;\;\;\;\;\;\;\;\;\;\;\;\;\;\;\;\;\;\;\; \text{ otherwise. } \end{cases}\]
We claim that 
	$$X \models \sigma \text{ \; iff \;} \sigma \in \Sigma.$$
If $\sigma \not\in \Sigma$, then $\sigma = \sigma_{2i} \in \Sigma^*$ for some $i < k$ and so the rows $2i$ and  $2i + 1$ of $X$ witness that $X \not\models \sigma$. 
Let $\sigma \in \Sigma$, we show that $X \models \sigma$. As in the proof of Theorem \ref{compl}, it suffices to show that $X \models \sigma$ for $\sigma = {{\Bumpeq}\mkern-2mu}(x_0, y_0 \cdots y_{m-1})$. Thus, suppose that $\sigma = {{\Bumpeq}\mkern-2mu}(x_0, y_0 \cdots y_{m-1})$. There are two cases.

\smallskip
\noindent {\bf Case 1.} ${{\Bumpeq}\mkern-2mu}(y_j, \emptyset) \in \Sigma$, for every $j < m$. Because of rule $(R_2)$ we have

\smallskip 

\begin{prooftree}
\AxiomC{$\Sigma$}
\UnaryInfC{${{\Bumpeq}\mkern-2mu}(y_0, \emptyset)$}
\AxiomC{$\Sigma$}
\UnaryInfC{$\phantom{(i)} \cdots \phantom{(i)}$}
\AxiomC{$\Sigma$}
\UnaryInfC{${{\Bumpeq}\mkern-2mu}(y_{m-1}, \emptyset)$}
\AxiomC{$\Sigma$}
\UnaryInfC{${{\Bumpeq}\mkern-2mu}(x_0, y_0 \cdots y_{m-1})$}
\LeftLabel{($R_2$)}
\QuaternaryInfC{${{\Bumpeq}\mkern-2mu}(x_0, \emptyset)$}
\end{prooftree}

\smallskip 

\noindent Thus, it is clear that $X \models \sigma$, because $x_0(s) = 0$ for every $s < n$.

\smallskip
 
\noindent {\bf Case 2.} There exists $j^* < m$ such that ${{\Bumpeq}\mkern-2mu}(y_{j^*}, \emptyset) \not\in \Sigma$. To show that $X \models \sigma$, assume that $x_0(s) \neq x_0(s')$ for some $s < s' < n$. Since ${{\Bumpeq}\mkern-2mu}(y_{j^*}, \emptyset) \not\in \Sigma$, we have:
$$ y_{j^*}(s) = \tau_{2i}(s)(y_{j^*}) \text{ and } y_{j^*}(s') = \tau_{2i'}(s')(y_{j^*}).$$
If $2i \neq 2{i'}$, then $y_{j^*}(s) \neq y_{j^*}(s')$ by definition and we are done.

Otherwise, if $2i = 2{i'}$, $s = 2i$ and $s' = 2i +1$, then, since $X_{2i} \models \sigma$, there exists $j < m$ such that $\tau_{2i}(s)(y_j) \neq \tau_{2i}(s')(y_j)$, which implies ${{\Bumpeq}\mkern-2mu}(y_{j}, \emptyset) \not\in \Sigma$ (as $X_{2i} \models \Sigma$). Thus,
$$ y_{j}(s) = \tau_{2i}(s)(y_j) \neq \tau_{2i}(s')(y_j)= y_{j}(s'),$$
and so we are done.

%
\end{proof}

\section{G-Dependence Logic}

	In this short section we give an explicit translation between the G-dependence atom and the dependence atom, and vice versa, thus answering one of V\"a\"an\"anen's questions from \cite{vaananen_G_dep} (see the question after Proposition~2). We use this translation and results from \cite[Chapter 6]{vaananen} to show that the expressive power of a natural logic built using G-dependence atoms (which we will call G-dependence logic) has the same expressive power as $\Sigma^1_{1}$, i.e. existential second-order logic. Notice that the axiomatization of $G$-dependence given in the previous section is not an axiomatization of G-dependence logic as defined below, but only an axiomatization of the G-dependence atoms. 
		
	\begin{definition}[Functional Dependence]\label{G_dep} Let $X = (I, \tau)$ be a team and $x$ and $y$ finite sets of variables from $\mathrm{dom}(X)$. We say that $y$ (functionally) depends on $x$, in symbols $X \models  \dep(x, y)$, if for every $s, s' \in I$, if $x(s) = x(s')$, then $y(s) = y(s')$.
\end{definition}



	\begin{proposition}\label{Jouko_prop_1} Let $X$ be a team and $x$ and $y$ finite sets of variables from $\mathrm{dom}(X)$. Then\footnote{The conjunctions below are taken in the metalanguage, they simply mean that for every $i < n$ the corresponding formula is true in X.},
	\begin{equation} X \models \dep(x, y) \;\; \Leftrightarrow \;\; \bigwedge_{i < n} X \models  \dep(x, y_i) \;\; \Leftrightarrow \;\; \bigwedge_{i < n} X \models  {{\Bumpeq}\mkern-2mu}(y_i, x),
\end{equation}
\begin{equation} X \models {{\Bumpeq}\mkern-2mu}(x, y) \;\; \Leftrightarrow \;\; \bigwedge_{i < n} X \models  {{\Bumpeq}\mkern-2mu}(x_i, (x-x_i)y) \;\; \Leftrightarrow \;\; \bigwedge_{i < n}  X \models \dep((x-x_i)y, x_i).
\end{equation}
\end{proposition}

	\begin{proof} Concerning (1), the first equivalence is immediate, while the second follows directly from \cite[Lemma 1, Item 5]{vaananen_G_dep}. Concerning (2), the first equivalence is because of the soundness of rules ($R_0$) and ($R'_0$), while the second follows again from \cite[Lemma 1, Item 5]{vaananen_G_dep}.
\end{proof}

	We now describe how to extend first-order logic into two new logics: {\em $G$-dependence logic} and {\em dependence logic}. 
	Given a vocabulary $L$, we define $\mathrm{GDep}(L)^-$ (resp. $\mathrm{Dep}(L)^-$) to be the set of atomic and negation of atomic $L$-formulas ($L$-literals), identity atoms and $G$-dependence atoms (resp. dependence atoms). We then define $\mathrm{GDep}(L)$ (resp. $\mathrm{Dep}(L)$) to be the language obtained closing $\mathrm{GDep}(L)^-$ (resp. $\mathrm{Dep}(L)^-$) under $\wedge$, $\vee$, $\exists$ and $\forall$ (no negations). Given $\phi \in \mathrm{GDep}(L)$ (or $\mathrm{Dep}(L)$) we denote by $\mathrm{Free}(\phi)$ the set of free variables occurring in $\phi$.
	
	\begin{definition}[Semantics]\label{semantics} Let $L$ be a vocabulary, $\mathcal{M}$ an $L$-structure, $\phi \in \mathrm{GDep}(L)$ (resp. $\phi \in \mathrm{Dep}(L)$) and $X = (I, \tau)$ a team with values in the domain of $\mathcal{M}$ and $\mathrm{Free}(\phi) \subseteq \mathrm{dom}(X)$. We define $\mathcal{M} \models_X \phi$ as follows:
	\begin{enumerate}[(1)]
	\item $\phi$ is ${{\Bumpeq}\mkern-2mu}(x, y)$ (resp. $\dep(x, y)$). Then $\mathcal{M} \models_X \phi$ iff $X \models {{\Bumpeq}\mkern-2mu}(x, y)$ (resp. $X \models~\dep(x, y))$.
	\item $\phi$ is a literal. Then $\mathcal{M} \models_X \phi$ iff for every $i \in I$, $\mathcal{M} \models_{\tau(i)} \phi$ in the usual first-order sense.
	\item $\phi$ is $\psi \wedge \theta$. Then $\mathcal{M} \models_X \phi$ iff $\mathcal{M} \models_X \psi$ and $\mathcal{M} \models_X \theta$.
	\item $\phi$ is $\psi \vee \theta$. Then $\mathcal{M} \models_X \phi$ iff there exist $J, K \subseteq I$ such that $J \cup K = I$ and $(J, \tau \restriction{J}) \models \psi$ and $(K, \tau \restriction{K}) \models \theta$ ($\restriction$ denotes restriction of functions as usual).
	\item $\phi$ is $\exists x \psi$. Then $\mathcal{M} \models_X \phi$ iff there exists a fucntion $F: I \rightarrow M$ such that $(I, \tau(F)) \models \psi$, where, for $i \in I$, $\tau(F)(i)(x) = F(i)$, and $\tau(F)(i)(y) = \tau(i)(y)$ for every $y \in \mathrm{dom}(X) - \{ x \}$.
	\item $\phi$ is $\forall x \psi$. Then $\mathcal{M} \models_X \phi$ iff $(I \times M, \tau(M)) \models \psi$, where, for $(i, a) \in I \times M$, $\tau(M)((i, a))(x) = a$, and $\tau(M)((i, a))(y) = \tau(i)(y)$ for every $y \in \mathrm{dom}(X) - \{ x \}$.
\end{enumerate}
\end{definition}

	Usually dependence logic and its variants are defined with respect to teams $X = (I, \tau)$ such that for $i \neq j \in I$ we have $\tau(i) \neq \tau(j)$. Clearly to every team $X$ as in Definition \ref{team} we can associate canonically a team $X'$ satisfying this additional assumption. Notice that if $\phi$ is as in Definition \ref{semantics} then $\mathcal{M} \models_X \phi$ iff $\mathcal{M} \models_{X'} \phi$.

	We now prove the announced characterization of the expressive power of $\mathrm{GDep}(L)$. For a quick introduction to second-order logic see e.g. \cite[Chapter 6]{vaananen}).

	\begin{corollary} For every vocabulary $L$ and G-dependence logic sentence $\phi$ in the vocabulary $L$, there exists a $\Sigma^1_{1}$ sentence $\Phi$ in the vocabulary $L$ such that for every $L$-structure $\mathcal{M}$ we have
		$$\mathcal{M} \models \phi \;\; \Leftrightarrow \;\;  \mathcal{M} \models \Phi.$$
Conversely, for every vocabulary $L$ and $\Sigma^1_{1}$ sentence $\Phi$ in the vocabulary $L$ there exists a G-dependence logic sentence $\phi$ in the vocabulary $L$ such that the above holds.
\end{corollary}

	\begin{proof} This follows from Proposition \ref{Jouko_prop_1} (and a straightforward induction), and Corollary 6.2 and Theorem 6.15 of \cite{vaananen}
\end{proof}

\end{document}